\documentclass[a4paper]{amsart}
\usepackage{amsmath,amssymb}
\input xy
\xyoption{all}

\numberwithin{equation}{section}
\newtheorem{thm}{Theorem}[section]
\newtheorem{lemma}[thm]{Lemma}

\theoremstyle{definition}
\newtheorem{exam}[thm]{Example}
\newtheorem{define}[thm]{Definition}
\newenvironment{quest}{\textbf{Question} }

\newcommand{\cd}{. }                                          
\newcommand{\kd}{, }                                          
\newcommand{\image}[1]{\mathrm{Im}\: #1}                      
\newcommand{\kernel}[1]{\mathrm{Ker}\: #1}

\newcommand{\socle}[1]{S(#1)}                    
              
\newcommand{\jac}[1]{J(#1)}

\newcommand{\tanni}[1]{\langle #1 \rangle}

\begin{document}
\title{Quivers with relations of Harada algebras}
\author{Kota Yamaura}
\address{Graduate School of Mathematics\\ Nagoya University\\ Frocho\\ Chikusaku\\ Nagoya\\ 464-8602\\ Japan}
\email{m07052d@math.nagoya-u.ac.jp}
\date{}
\maketitle

\begin{abstract}
For a finite dimensional algebra $R$\kd we give an explicit
description of quivers with relations of block extension of
$R$\cd As an application\kd we describe quivers with relations of Harada
algebras by using those of the corresponding quasi-Frobenius algebras\cd
\end{abstract}

\section{Introduction}
Two classes of artinian rings have been studied for a long
time\cd
The one is Nakayama rings and the other is quasi-Frobenius rings
(QF-rings)\cd
In 1980s\kd Harada introduced the new class of artinian rings called Harada
rings nowadays\kd
which give a common generalization of Nakayama rings and QF-rings\cd
Many authors have studied the structure of Harada rings (e.g. \cite{quasi_H,Lec_N,dua_H,LE_QF,LE_gur,OHr1,OHr2,OHr3})\cd
Among others, Oshiro \cite{OHr3} gave a structure theorem of Harada rings as upper
staircase factor rings of block extensions of QF-rings\cd
This plays an important role in the theory of Harada rings\cd

In \cite{Thr}, Thrall paid attention to three properties of QF-algebras, called QF-$1$, QF-$2$, and QF-$3$. 
Harada algebras satisfies the property QF-$3$ which is the condition that injective hull of the algebra 
is projective. 
This property is often called $1$-Gorenstein \cite{FGR,AR} or dominant dimension at least one \cite{T}, 
and often plays an important role in the representation theory. 
Harada algebras give a class of QF-$3$ algebras, and their indecomposable projective modules 
have "nice" structure (see Definition \ref{block_ext}).

In this paper\kd we study Harada algebras by using method of representation
theory of algebras\cd
In 1970s\kd Gabriel introduced quivers with relations in representation
theory of algebras\cd
This gives a powerful tool to calculate modules over algebras (see
\cite{ASS,ARS})\cd 
The aim of this paper is to describe the quivers with relations of Harada
algebras by using those of the corresponding QF-algebras\cd
More generally\kd we give a description of quivers with relations of block
extensions of arbitrary finite dimensional algebras\cd

This paper is organized as follows\cd In section 2\kd we recall basic
results on Harada algebras\kd especially their structure theorem in terms of
block extensions and upper staircase factor algebras\cd
In section 3\kd we state our main theorems that describe the quivers with
relations of block extensions of arbitrary algebras\cd
We prove them in section 4\cd
In section 5\kd we describe quivers with relations of Harada algebras by
studying those of upper staircase factor algebras\cd

Throughout this paper\kd an algebra means a finite dimensional associative
algebra over an algebraically closed field $K$\cd
We always deal with right $R$-modules\cd
We denote by $\jac{M}$ the Jacobson radical of an $R$-module $M$ and by
$\socle{M}$ the socle of $M$\cd
We denote by $\tanni{x}_{ij}$ the matrix with the $(i,j)$-entry $x$ and
other entries $0$\cd
\bigskip

\section{Block extension and Harada algebras}
In this section\kd we recall definitions and basic facts on Harada algebras\kd in particular block extensions\kd upper staircase factor algebras\kd and the structure theorem of Harada algebras\cd These are due to Oshiro in \cite{OHr1} \cite{OHr2} \cite{OHr3}\cd

We start with giving a definition of left Harada algebra from its structural point of view\cd 

\begin{define}\label{def_Harada}
Let $R$ be a basic algebra and $\mathrm{Pi}(R)$ be a complete set of
orthogonal primitive idempotents of $R$\cd
We call $R$ a \emph{left Harada algebra} if $\mathrm{Pi}(R)$ can be arranged that $\mathrm{Pi}(R)=\{e_{ij}\}_{i=1}^{m},_{j=1}^{n_i}$
where
\begin{enumerate}
\def\labelenumi{(\theenumi)}
\item $e_{i1}R$ is an injective $R$-module for any $i=1,\cdots,m$\kd
\item $e_{ij}R \simeq e_{i,j-1}\jac{R}$ for any $i=1,\cdots,m \kd
j=2,\cdots,n_i$\cd
\end{enumerate}
\end{define}

In the rest of this section\kd we recall the structure theorem of Harada algebras\cd
First we have to recall block extensions\cd

\begin{define}\label{block_ext}
Let $R$ be a basic algebra and let $\mathrm{Pi}(R)=\{e_1,\cdots,e_m\}$\cd
Then $R$ can be represented as the following matrix form\cd
\begin{eqnarray*}
R=\left( \begin{array}{cccc}
e_1Re_1 & e_1Re_2 & \cdots & e_1Re_m \\
e_2Re_1 & e_2Re_2 & \cdots & e_2Re_m \\
\vdots & \vdots & \ddots & \vdots \\
e_mRe_1 & e_mRe_2 & \cdots & e_mRe_m
\end{array} \right)
\end{eqnarray*}
We put $Q_i=e_iRe_i$ and $A_{ij}=e_iRe_j$\cd
Then $Q_i$ is the subalgebra of $R$ with an identity element $e_i$\kd and $A_{ij}$ is a $(Q_i,Q_j)$-bimodule\cd

Let $n_1,n_2,\cdots,n_m \in \mathbb{N}$ be natural numbers\cd
For $1 \leq i,s \leq m \kd 1 \leq j \leq n_i$ and $1 \leq t \leq n_s$\kd we define $P_{ij,st}$ as follows\cd
\[
P_{ij,st}:=
\begin{cases}
Q_i & (i=s \kd j \leq t)  \\
\jac{Q_i} & (i=s \kd j>t) \\
A_{ij} & (i \neq s) \\
\end{cases}
\]
and put $n_i \times n_s$ matrices $P(i,s)$
\begin{eqnarray*}
P(i,s) :=
\left( \begin{array}{cccc}
P_{i1,s1} & P_{i1,s2} & \cdots & P_{i1,sn_s} \\
P_{i2,s1} & P_{i2,s2} & \cdots & P_{i2,sn_s} \\
\vdots  & \vdots  & \ddots & \vdots  \\
P_{in_i,s1} & P_{in_i,s2} & \cdots & P_{in_i,sn_s}
\end{array} \right) 
= \begin{cases}
\left( \begin{array}{ccc}
Q_i & \cdots & Q_i \\
   & \ddots & \vdots  \\
\jac{Q_i} &  & Q_i
\end{array} \right) & (i=s) \\
\left( \begin{array}{ccc}
A_{is} & \cdots & A_{is} \\
\vdots &        & \vdots \\
A_{is} & \cdots & A_{is}
\end{array} \right) & (i \neq s) \cd
\end{cases}
\end{eqnarray*}
We define the \emph{block extension} $P$ of $R$ for $\{ n_1,\cdots,n_m  \}$ as
\[
P=R(n_1,\cdots,n_m)=\left( \begin{array}{cccc}
P(1,1) & P(1,2) & \cdots & P(1,m) \\
P(2,1) & P(2,2) & \cdots & P(2,m) \\
\vdots & \vdots & \ddots & \vdots \\
P(m,1) & P(m,2) & \cdots & P(m,m)
\end{array} \right) \cd
\]
Clearly using usual matrix multiplication\kd $P$ forms a ring\cd
\end{define}

We can see that $P$ is a basic finite dimensional $K$-algebra with a complete set of orthogonal primitive idempotents $\{ f_{ij} \ | \ 1 \leq i \leq m \kd 1 \leq j \leq n_i \}$ where $f_{ij}=\tanni{1}_{ij,ij}$\cd
It is known that any block extension of a basic QF-algebra is a basic left Harada algebra (see \cite{OHr1})\cd

\begin{exam}\label{exam1}
For example\kd let $m=2$ and consider the case $n_1=3 \kd n_2=2$\cd
Then
\[
R(3,2)=
\left( \begin{array}{ccc|cc}
Q_1 & Q_1 & Q_1 & A_{12} & A_{12} \\
\jac{Q_1} & Q_1 & Q_1 & A_{12} & A_{12} \\
\jac{Q_1} & \jac{Q_1} & Q_1 & A_{12} & A_{12} \\ \hline
A_{21} & A_{21} & A_{21} & Q_2 & Q_2 \\
A_{21} & A_{21} & A_{21} & \jac{Q_2} & Q_2 \\
\end{array} \right) \cd
\]
This is a Harada algebra if $R$ is a QF-algebra\cd
\end{exam}

Next we recall upper staircase factor algebras of block extensions\cd

\begin{define}\label{upper_st_fac}
Let $R$ be a basic QF-algebra over a field $K$ with Nakayama permutation $\sigma$ and $P=R(n_1,\cdots,n_m)$ a block extension in Definition \ref{block_ext}\cd 
Then $P$ is a basic left Harada algebra\cd
First we define a sub $(P(i,i),P(\sigma(i),\sigma(i)))$-module $S(i,\sigma(i))$ of $P(i,\sigma(i))$ as follows\cd
\begin{itemize}
\item[(I)] Case $i=\sigma(i)$\cd We take a sequence
\[
1 \leq c_{i1} \leq c_{i2} \leq \cdots \leq c_{in_i} \leq n_i \cd
\]
of natural numbers\cd
We define a subset $S(i,i)$ of $P(i,i)$ as follows\cd
\[
S_{ij,it}=\begin{cases}
0 & (1 \leq t \leq c_{ij}) \kd \\
\socle{Q_i} & (c_{ij} < t \leq n_i)\kd
\end{cases}
\]
\[
P(i,i) \supset S(i,i)=\left( \begin{array}{cccc}
S_{i1,i1} & S_{i1,i2} & \cdots & S_{i1,in_i} \\
S_{i2,i1} & S_{i2,i2} & \cdots & S_{i2,in_i} \\
\vdots  & \vdots & \ddots & \vdots  \\
S_{in_i,i1} & S_{in_i,i2} & \cdots & S_{in_i,in_i}
\end{array} \right) \cd
\]
We note that $\socle{Q_i}$ is a left simple and a right simple $Q_i$-module\cd
We can see that $S(i,i)$ is sub $(P(i,i),P(i,i))$-module of $P(i,i)$\cd 
It has the following form\cd
\[
S(i,i) =
\left(
\begin{array}{@{\,}ccccc@{\,}}
 0  \cdots  0  &\multicolumn{3}{|r}{\hspace{10mm}S(Q_i)}  \\ \cline{2-2}
  &    &  \multicolumn{1}{|c}{} \\
  &    &  \multicolumn{1}{|c}{} \\ \cline{3-3}
  &  &  &  \multicolumn{1}{|c}{}  \\ \cline{4-4}
 \multicolumn{3}{l}{0}&    \\
\end{array}\right) \cd
\]
\item[(II)] Case $i \neq \sigma(i)$\cd We take a sequence
\[
1 \leq c_{i1} \leq c_{i2} \leq \cdots \leq c_{in_i} \leq n_{\sigma(i)} \cd
\]
of natural numbers\cd
We define a subset $S(i,\sigma(i))$ of $P(i,\sigma(i))$ as follows\cd
\[
S_{ij,\sigma(i)t}=\begin{cases}
0 & (1 \leq t \leq c_{ij}) \kd \\
\socle{A_{i\sigma(i)}} & (c_{ij} < t \leq n_{\sigma(i)}) \kd
\end{cases}
\]
\[
P(i,\sigma(i)) \supset S(i,\sigma(i))=\left( \begin{array}{cccc}
S_{i1,\sigma(i)1} & S_{i1,\sigma(i)2} & \cdots & S_{i1,\sigma(i)n_{\sigma(i)}} \\
S_{i2,\sigma(i)1} & S_{i2,\sigma(i)2} & \cdots & S_{i2,\sigma(i)n_{\sigma(i)}} \\
\vdots  & \vdots & \ddots & \vdots  \\
S_{in_i,\sigma(i)1} & S_{in_i,\sigma(i)2} & \cdots & S_{in_i,\sigma(i)n_{\sigma(i)}}
\end{array} \right) \cd
\]
We note that $\socle{A_{i\sigma(i)}}$ is a left simple $Q_i$-module and a right simple $Q_{\sigma(i)}$-module\cd
We can see that $S(i,\sigma(i))$ is sub $(P(i,i),P(\sigma(i),\sigma(i)))$-module of $P(i,\sigma(i))$\cd
It has the following form\cd
\[
S(i,\sigma (i)) =
\left(
\begin{array}{@{\,}ccccc@{\,}}
 0  \cdots  0  &\multicolumn{3}{|r}{\hspace{10mm}S(A_{i \sigma{i}})}  \\ \cline{2-2}
  &    &  \multicolumn{1}{|c}{} \\
  &    &  \multicolumn{1}{|c}{} \\ \cline{3-3}
  &  &  &  \multicolumn{1}{|c}{}  \\ \cline{4-4}
 \multicolumn{3}{l}{0}&    \\
\end{array}\right) \cd
\]
\end{itemize}

Next we define a subset $X$ of $P$ by putting
\[
P(i,s) \supset X(i,s) = \begin{cases}
S(i,\sigma(i)) & (s=\sigma(i)) \kd \\
0              & (s \neq \sigma(i)) \kd
\end{cases}
\]
\[
X=\left( \begin{array}{cccc}
X(1,1) & X(1,2) & \cdots & X(1,m) \\
X(2,1) & X(2,2) & \cdots & X(2,m) \\
\vdots & \vdots & \ddots & \vdots \\
X(m,1) & X(m,2) & \cdots & X(m,m)
\end{array} \right) \cd
\]
Then $X$ is an ideal of $P$\cd We define an \emph{upper staircase factor algebra} $\overline{P}$ of $P$ by
\[
\overline{P}=P/X=R(n_1,\cdots,n_m)/X \cd
\]
\end{define}

It is known that any upper staircase factor algebra of a block extension of a basic QF-algebra is a basic left Harada algebra (see \cite{OHr3})\cd
Conversely all left Harada algebras are constructed by these operations\cd
In fact\kd the following well-known result holds\cd

\begin{thm}[Structure Theorem of Harada algebras \cite{OHr3}]\label{Str_of_H}
For any basic left Harada algebra $T$\kd there exists a basic QF-algebra $R$ such that $T$ is isomorphic to an upper staircase factor algebra of a block extension of $R$\cd
\end{thm}
We call the above $R$ the \emph{frame QF-algebra} of $T$\cd

The above theorem motivates us to consider the following question: \\

\begin{quest}
Is the quiver with relations of a Harada algebra described by using those of the frame QF-algebra ?
\end{quest}
\bigskip

The answer is ``yes''\cd We describe the quivers with relations of block extensions of arbitrary (not necesarilly QF) algebras $R$ by using those of $R$ in section $3$ and $4$\cd
Moreover for the case $R$ is QF\kd we describe the quivers with relations of upper staircase factor algebras in section $5$\cd

\begin{exam}\label{exam_of_str}
We consider $P=R(3,2)$ of Example \ref{exam1}\cd
We assume that $R$ is a basic QF-algebra with Nakayama permutation identity\cd
For example\kd we take the following sequence\cd
\[
c_{11}=1 \kd c_{12}=2 \kd c_{13}=2 \kd c_{21}=1 \kd c_{22}=2 \cd
\]
Then $X$ defined above is as follows\cd
\[
X=\left( \begin{array}{ccc|cc}
0 & \socle{Q_1} & \socle{Q_1} & 0 & 0 \\
0 & 0 & \socle{Q_1} & 0 & 0 \\
0 & 0 & \socle{Q_1} & 0 & 0 \\ \hline
0 & 0 & 0 & 0 & \socle{Q_2} \\
0 & 0 & 0 & 0 & 0
\end{array} \right) \cd
\]
The factor algebra $P/X$ is a left Harada algebra\cd
\end{exam}
\bigskip

\section{Main theorem}
In this section\kd we first recall basic definitions on presentations of algebras by using quivers with relations 
(we refer to \cite[ChapterII]{ASS} for the detail)\cd
Next we state the main result of this paper which describes 
quivers with relations of block extension algebras of $R$ by using those of $R$\cd

Let $R$ be a basic finite dimensional algebra over an algebraically closed field $K$\cd
We use the notations in Definition \ref{block_ext}\cd 
The Jacobson radical of $R$ is given by
\[
\jac{R}=\left( \begin{array}{cccc}
\jac{Q_1} & A_{12} & \cdots & A_{1m} \\
A_{21}  & \jac{Q_2} & \cdots & A_{2m} \\
\vdots & \vdots & \ddots & \vdots \\
A_{m1} & A_{m2} & \cdots & \jac{Q_m}
\end{array}
\right)
\]
(see \cite{AF})\cd
Hence we have
\begin{equation}\label{J(R)^2}
\jac{R}^2=\left( \begin{array}{ccc}
X_{11} & \cdots & X_{1m} \\
\vdots &        & \vdots \\
X_{m1} & \cdots & X_{mm}
\end{array} \right) \kd
\end{equation}
where
\[
X_{ij}= \begin{cases}
\jac{Q_i}^2+\sum_{k \neq i} A_{ik}A_{ki} & (i=j) \kd \\
\jac{Q_i}A_{ij}+A_{ij}\jac{Q_j}+\sum_{k \neq i,j} A_{ik}A_{kj} & (i \neq j) \cd
\end{cases}
\]
We put
\[
d_{ij}= \begin{cases}
\dim_K \jac{Q_i}/X_{ii} & (i=j) \kd \\
\dim_K A_{ij}/X_{ij} & (i \neq j) \cd
\end{cases}
\]
Then the quiver $Q$ of $R$ is defined as follows\cd
\begin{itemize}
\item[(a)] The set of vertices of $Q$ is $\{ 1,2,\cdots,m \}$\cd
\item[(b)] There are $d_{ij}$ arrows $\{ \alpha^1_{ij},  \alpha^2_{ij},\cdots,\alpha^{d_{ij}}_{ij} \}$ from $i$ to $j$\cd
\end{itemize}

We define the path algebra of $Q$\cd Let $KQ$ be a $K$-vector space with the basis consisting of all paths in $Q$\cd
We define multiplication of $KQ$ as follows\cd
For any paths $\alpha=\alpha_1 \cdots \alpha_s \kd \beta=\beta_1 \cdots \beta_t \ (\alpha_k \kd \beta_l$ are arrows)\kd
we put $i$ the end vertex of $\alpha_k$ and put $j$ the starting vertex of $\beta_1$\kd then we define
\[
\alpha \beta = \begin{cases}
\alpha_1 \cdots \alpha_s \beta_1 \cdots \beta_t & (i=j) \kd \\
0 & (i \neq j) \cd
\end{cases}
\]
We extend the multiplication linearly to $KQ$\cd 
Then we have a $K$-algebra $KQ$ called the \emph{path algebra} of $Q$\cd
We call $p \in KQ$ \emph{basic} if all paths appearing in $p$ have the same source and target\cd

We fix a basis $\{x_{ij}^1 \kd x_{ij}^2 \kd \cdots \kd x_{ij}^{d_{ij}}\}$ of $K$-vector spaces $\jac{Q_i}/X_{ii}$ and $A_{ij}/X_{ij}$\cd
By the universal property of path algebras\kd we can define a $K$-algebra homomorphism $\varphi:KQ \longrightarrow R$ by
\begin{eqnarray*}
kQ \ni i & \longmapsto & \tanni{1}_{ii} \in R \kd \\
kQ \ni \alpha^t_{ij}  & \longmapsto & \tanni{x^t_{ij}}_{ij} \in R \cd
\end{eqnarray*}
Then $\varphi$ is surjective\cd
Consequently\kd we have a $K$-algebra isomorphism
\[
KQ/\kernel{\varphi} \simeq R \cd
\]
We fix a set of generators $\{ \rho_1,\cdots,\rho_w \}$ of the two-sided ideal $\kernel{\varphi}$ of $R$\cd
We can assume that $\rho_i$ is basic for any $i$\cd

Now we state our main results which give a quiver with relations of a block extension algebra of $R$\cd

\begin{thm}\label{main_thm}
Let $n_1,\cdots,n_m \in \mathbb{N}$ and put $P=R(n_1,\cdots,n_m)$\cd
The quiver $Q'$ of $P$ is given as follows\cd
\begin{itemize}
\item[(a)] The set of vertices of $Q'$ is $\{ (i,j) \ | \ 1 \leq i \leq m \kd 1 \leq j \leq n_i \}$\cd
\item[(b)] There are the following two kinds of arrows\cd
\begin{list}{$\bullet$}{}
\item An arrow $\delta_{ij}$ from $(i,j)$ to $(i,j+1)$\cd
\item $d_{is}$ arrows $\{ \beta^1_{is},\beta^2_{is}, \cdots,  \beta^{d_{is}}_{is} \}$ from $(i,n_i)$ to $(s,1)$\cd
\end{list}
\end{itemize}
\end{thm}

To describe relations for $P$\kd we have to define the extension map\cd

\begin{define}\label{extension_map}
Let $KQ_+$ (respectively\kd $KQ_+'$) be a $K$-subspace of $KQ$ (respectively\kd $KQ'$) generated by all paths of length
$\geq 1$\cd
Using $\delta_{ij}$ defined in Theorem \ref{main_thm}\kd we put
\[
\delta_i = \delta_{i1} \delta_{i2} \cdots \delta_{in_i-1} \in KQ' \cd
\]

We define a $K$-linear map $e:KQ_+ \longrightarrow KQ_+'$ called the \emph{extension map}\kd as follows\cd
Any path
\[
p=\alpha^{t_1}_{i_1i_2} \alpha^{t_2}_{i_2i_3} \cdots
\alpha^{t_k}_{i_ki_{k+1}}
\]
of $KQ$ corresponds to the path
\[
e(p)=\beta^{t_1}_{i_1i_2} \delta_{i_2}  \beta^{t_2}_{i_2i_3} \cdots
\cdots \beta^{t_{k-1}}_{i_{k-1}i_{k}} \delta_{i_k}
\beta^{t_k}_{i_ki_{k+1}}
\]
of $KQ'$\cd
Clearly $e$ is an injection\cd
\end{define}

Then we can describe relations for $P$\cd

\begin{thm}\label{main_thm2}
Under the hypotheses of Theorem \ref{main_thm}\kd we have a $K$-algebra isomorphism
\[
P \simeq KQ'/(e(\rho_1),\cdots,e(\rho_w))\cd
\]
\end{thm}

\begin{exam}\label{exam_main_thm}
Let $R$ be a $K$-algebra given by the following quiver with relations\cd
\[
\xymatrix{
1 \ar@(dl,ul)[]^{\alpha_{11}}   \ar@<0.4ex>[rr]^{\alpha_{12}} & & 2 \ar@<0.4ex>[ll]^{\alpha_{21}}
}
\quad \quad
\begin{cases}
\alpha^3_{11}=\alpha_{12}\alpha_{21} \\
\alpha_{11} \alpha_{12} = 0 \\
\alpha_{21} \alpha_{11} = 0
\end{cases}
\]
Then $R$ can be represented as the following matrix form\cd
\[
R = \left( \begin{array}{cc}
K[\alpha_{11}]/(\alpha_{11})^4 & K \alpha_{12} \\
K  \alpha_{21} & K[\alpha_{21}\alpha_{12}]/(\alpha_{21}\alpha_{12})^2
\end{array} \right)
\]

For example\kd $R(3,2)$ can be represented as the followng matrix form\cd
\[
{\tiny
R(3,2) = \left( \begin{array}{ccc|cc}
K[\alpha_{11}]/(\alpha_{11})^4 & K[\alpha_{11}]/(\alpha_{11})^4 & K[\alpha_{11}]/(\alpha_{11})^4
& K \alpha_{12} & K \alpha_{12}  \\
(\alpha_{11})/(\alpha_{11})^4 & K[\alpha_{11}]/(\alpha_{11})^4 & K[\alpha_{11}]/(\alpha_{11})^4
& K \alpha_{12} & K \alpha_{12}  \\
(\alpha_{11})/(\alpha_{11})^4 & (\alpha_{11})/(\alpha_{11})^4 & K[\alpha_{11}]/(\alpha_{11})^4
& K \alpha{12} & K \alpha{12}  \\ \hline
K \alpha_{21} & K \alpha_{21} & K \alpha_{21}
& K[\alpha_{21}\alpha_{12}]/(\alpha_{21}\alpha_{12})^2 & K[\alpha_{21}\alpha_{12}]/(\alpha_{21}\alpha_{12})^2 \\
K \alpha_{21} & K \alpha_{21} & K \alpha_{21}
& (\alpha_{21}\alpha_{12})/(\alpha_{21}\alpha_{12})^2 & K[\alpha_{21}\alpha_{12}]/(\alpha_{21}\alpha_{12})^2
\end{array} \right) \cd
}
\]
The quiver with relations of $R(3,2)$ is given by following\cd
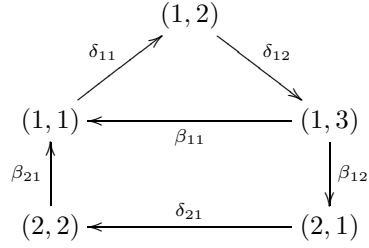
\[
\xymatrix{
& (1,2) \ar[rd]^{\delta_{12}} & \\
(1,1) \ar[ru]^{\delta_{11}} & & (1,3) \ar[d]^{\beta_{12}} \ar[ll]^{\beta_{11}}  \\
(2,2) \ar[u]^{\beta_{21}} & & (2,1) \ar[ll]_{\delta_{21}}
}
\]
\[
\begin{cases}
e(\alpha^3_{11})=\beta_{11} \delta_{11} \delta_{12} \beta_{11} \delta_{11} \delta_{12} \beta_{11}=\beta_{12} \delta_{21} \beta_{21}=e(\alpha_{12}\alpha_{21}) \\
e(\alpha_{11} \alpha_{12})=\beta_{11} \delta_{11} \delta_{12} \beta_{12} = 0 \\
e(\alpha_{21} \alpha_{11})=\beta_{21} \delta_{11} \delta_{12} \beta_{11} = 0
\end{cases}
\]
\end{exam}
\bigskip

\section{Proof of Main theorem}
In this section\kd we keep the notations in the previous section\cd 
We prove the Theorem \ref{main_thm} and \ref{main_thm2}\cd

Let $P$ be a block extension $R(n_1,\cdots,n_m)$\cd 
The Jacobson radical of $P$ is given by
\[
\jac{P}=\left( \begin{array}{cccc}
J(1,1) & J(1,2) & \cdots & J(1,m) \\
J(2,1) & J(2,2) & \cdots & J(2,m) \\
\vdots & \vdots & \ddots & \vdots \\
J(m,1) & J(m,2) & \cdots & J(m,m)
\end{array} \right) \cd
\]
\[
P_{ij,st} \supset J_{ij,st}= \begin{cases}
Q_i & (i=s \kd j+1 \leq t) \kd \\
\jac{Q_i} & (i=s \kd j+1>t) \kd \\
A_{is} & (i \neq s) \kd
\end{cases}
\]
where we use the notation in Definition \ref{block_ext}\cd
\begin{eqnarray*}
P(i,s) \supset J(i,s) &=&
\left( \begin{array}{ccc}
J_{i1,s1} & \cdots & J_{i1,sn_s} \\
\vdots & & \vdots \\
J_{in_i,s1} & \cdots & J_{in_i,sn_s}
\end{array} \right) \\
&=& \begin{cases}
\left( \begin{array}{ccccc}
\jac{Q_i} & Q_i & \cdots & \cdots & Q_i \\
 & \jac{Q_i} & Q_i & \cdots & Q_i \\
 &  & \ddots & \ddots & \vdots \\
 &  &  & \jac{Q_i} & Q_i\\
\jac{Q_i} &  &  &  & \jac{Q_i}
\end{array} \right) & (i=s) \kd \\
\left( \begin{array}{ccc}
A_{is} & \cdots & A_{is} \\
\vdots &        & \vdots \\
A_{is} & \cdots & A_{is}
\end{array} \right) & (i \neq s) \cd
\end{cases}
\end{eqnarray*}
Now let us calculate $\jac{P}^2$\cd We denote by $Y_{ij,st}$ the $(ij,st)$-entry of $\jac{P}^2$\cd \\
(I) Case $i=s$\cd
\begin{eqnarray*}
Y_{ij,it} &=& \sum_{k=1}^{m} \sum_{l=1}^{n_k} J_{ij,kl} J_{kl,it} \\
&=& \sum_{l=1}^{n_i} J_{ij,il} J_{il,it} + \sum_{k=1,k \neq i}^{m}
\sum_{l=1}^{n_k} J_{ij,kl} J_{kl,it} \\
&=& \begin{cases}
Q_i & (j+1 < t) \kd \\
\jac{Q_i} & (j+1 \geq t\kd (j,t) \neq (n_i,1)) \kd \\
\jac{Q_i}^2+\sum_{k=1,k \neq i}^{m} A_{ik}A_{ki} = X_{ii} & ((j,t)=(n_i,1)) \kd
\end{cases}
\end{eqnarray*}
where $X_{ii}$ is in \eqref{J(R)^2}\cd  \\
(II) Case $i \neq s$\cd
\begin{eqnarray*}
Y_{ij,st} &=& \sum_{k=1}^{m} \sum_{l=1}^{n_k} J_{ij,kl} J_{kl,st} \\
&=& \sum_{l=1}^{n_i} J_{ij,il} J_{il,st} + \sum_{l=1}^{n_j} J_{ij,jl}
J_{jl,st}  + \sum_{k=1,k \neq i,j}^{m} \sum_{l=1}^{n_k} J_{ij,kl}
J_{kl,st} \\
&=& \begin{cases}
A_{is} & ((j,t) \neq (n_i,1)) \kd \\
\jac{Q_i}A_{is}+A_{is}\jac{Q_s}+\sum_{k=1,k \neq i,j}^{m} A_{ik}A_{ks}=X_{is} & ((j,t)=(n_i,1)) \cd
\end{cases}
\end{eqnarray*}
where $X_{is}$ is in \eqref{J(R)^2}\cd  \\

Therefore
\begin{eqnarray*}
P(i,s) \supset Y(i,s) &=&
\left( \begin{array}{ccc}
Y_{i1,s1} & \cdots & Y_{i1,sn_s} \\
\vdots & & \vdots \\
Y_{in_i,s1} & \cdots & Y_{in_i,sn_s}
\end{array} \right) \\
&=& \begin{cases}
\left( \begin{array}{ccccc}
\jac{Q_i} & \jac{Q_i} &  &  & Q_i \\
\vdots  & \jac{Q_i} & \jac{Q_i} &  & \\
\vdots  &  & \ddots & \ddots &  \\
\jac{Q_i}  & \jac{Q_i} &  & \jac{Q_i} & \jac{Q_i} \\
X_{ii} & \jac{Q_i}  & \cdots & \cdots  & \jac{Q_i}
\end{array} \right) & (i=s) \kd \\
\left( \begin{array}{cccc}
A_{is} & A_{is} & \cdots & A_{is} \\
\vdots & \vdots &        & \vdots \\
A_{is} & A_{is} & \cdots & A_{is} \\
X_{ij} & A_{is} & \cdots & A_{is}
\end{array} \right) & (i \neq s) \cd
\end{cases}
\end{eqnarray*}
\[
\jac{P}^2=\left( \begin{array}{cccc}
Y(1,1) & Y(1,2) & \cdots & Y(1,m) \\
Y(2,1) & Y(2,2) & \cdots & Y(2,m) \\
\vdots & \vdots & \ddots & \vdots \\
Y(m,1) & Y(m,2) & \cdots & Y(m,m)
\end{array} \right) \cd
\]
Thus we have
\[
J(i,s)/Y(i,s)=\begin{cases}
\left( \begin{array}{cccccc}
0  & Q_i/\jac{Q_i} & 0  & \cdots  & 0  \\
\vdots  & 0  & Q_i/\jac{Q_i} & 0 &   \vdots  \\
\vdots  &  & \ddots & \ddots  & 0  \\
0  & 0  &   & 0 & Q_i/\jac{Q_i} \\
\jac{Q_i}/X_{ii} & 0  & \cdots & \cdots  & 0
\end{array} \right) & (i=s) \kd \\
\left( \begin{array}{cccc}
0 & 0  & \cdots & 0  \\
\vdots & \vdots &  & \vdots \\
0 & 0  & \cdots & 0  \\
A_{is}/X_{is} & 0 & \cdots & 0
\end{array} \right) & (i \neq s) \cd
\end{cases}
\]
Consequently the quiver $Q'$ of $P$ is given by Theorem \ref{main_thm}\cd \hfill{$\square$}
\bigskip

Next\kd we shall prove Theorem \ref{main_thm2}\cd 
Define a $K$-algebra homomorphism $\varphi':KQ' \longrightarrow P$ by
\begin{eqnarray*}
KQ' \ni (i,j) & \longmapsto & \tanni{1}_{ij,ij} \in P \kd \\
KQ' \ni \delta_{ij}  & \longmapsto & \tanni{1}_{ij,ij+1} \in P \kd \\
KQ' \ni \beta^t_{is} & \longmapsto & \tanni{x^t_{is}}_{in_i,s1} \in P \cd
\end{eqnarray*}
Then $\varphi'$ is surjective\cd Consequently\kd we have a $K$-algebra isomorphism
\[
KQ'/\kernel{\varphi'} \simeq P \cd
\]
We only have to show that the two-sided ideal $\kernel{\varphi'}$ of $P$ is generated by the $e(\rho_1),\cdots,e(\rho_w)$\cd
Let us start with giving elementary properties of the extension map\cd

\begin{lemma}\label{lem_of_ext_map}
For the extension map $e:KQ_+ \longrightarrow KQ_+'$\kd the following hold\cd
\begin{enumerate}
\def\labelenumi{(\theenumi)}
\item The two-sided ideal $\kernel{\varphi'}$ of $P$ is generated by elements in $\image{e}$\cd
\item For any $p \in KQ_+$\kd $\varphi(p)=0$ if and only if
$\varphi'(e(p))=0$\cd
\item For any path $p \in KQ_+$ ending at $j$ and any path $q \in KQ_+$ starting at $j$\kd we have
\[
e(pq)=e(p)\delta_{j} e(q)
\]
where $\delta_{j}$ is defined in Definition \ref{extension_map}\cd
\end{enumerate}
\end{lemma}
\begin{proof}
(1) We take $p' \in \kernel{\varphi'}$\cd We can assume that $p'$ is basic\cd 
By definition of quiver $Q'$\kd we can write $p'$ as follows\cd
\[
p'=\delta_{ij} \cdots \delta_{in_i} p'' \delta_{s1} \cdots \delta_{st} \quad (p'' \in \image{e})\cd
\]
We have
\begin{eqnarray*}
\varphi'(p') &=& \varphi'(\delta_{ij} \cdots \delta_{in_i} p'' \delta_{s1} \cdots \delta_{st})\\
&=& \varphi'(\delta_{ij}) \cdots \varphi'(\delta_{in_i}) \varphi'( p'') \varphi'(\delta_{s1}) \cdots \varphi'(\delta_{st}) \\
&=& \tanni{1}_{ij,ij+1} \cdots \tanni{1}_{in_i-1,in_i} \varphi'( p'') 
\tanni{1}_{s1,s2} \cdots \tanni{1}_{st,st+1}\cd
\end{eqnarray*}
We have $\varphi'(p'')=0$ since $\varphi'(p')=0$\cd 
Consequently $p'' \in \kernel{\varphi'} \cap \image{e}$\cd
Thus the assertion follows\cd

(2) We take $KQ_+ \ni p=\alpha^{t_1}_{i_1i_2} \alpha^{t_2}_{i_2i_3}
   \cdots \alpha^{t_k}_{i_ki_{k+1}}$\cd Then we have
\begin{eqnarray*}
\varphi(p)&=&\tanni{x^{t_1}_{i_1i_2}}_{i_1i_2}
   \tanni{x^{t_2}_{i_2i_3}}_{i_2i_3} \cdots
   \tanni{x^{t_k}_{i_ki_{k+1}}}_{i_ki_{k+1}} \\
&=&\tanni{x^{t_1}_{i_1i_2} x^{t_2}_{i_2i_3} \cdots
   x^{t_k}_{i_ki_{k+1}} }_{i_1i_{k+1}} \cd
\end{eqnarray*}
On the other hand\kd we have
\begin{eqnarray*}
\varphi'(e(p))
&=& \varphi'(\beta^{t_1}_{i_1i_2} \delta_{i_2} \beta^{t_2}_{i_2i_3} \cdots \beta^{t_k}_{i_ki_{k+1}}) \\
&=& \varphi'(\beta^{t_1}_{i_1i_2}) \varphi'(\delta_{i_2}) \varphi'( \beta^{t_2}_{i_2i_3}) 
\cdots \varphi'(\beta^{t_k}_{i_ki_{k+1}}) \\
&=& \tanni{x^{t_1}_{i_1i_2}}_{i_1n_{i_1},i_21} \tanni{1}_{i_21,i_2n_{i_2}} 
\tanni{x^{t_2}_{i_2i_3}}_{i_2n_{i_2},i_3 1} \cdots \tanni{x^{t_k}_{i_ki_{k+1}}}_{i_kn_{i_k},i_{k+1} 1} \\
&=& \tanni{x^{t_1}_{i_1i_2} x^{t_2}_{i_2i_3} \cdots x^{t_k}_{i_ki_{k+1}} }_{i_1n_{i_1},i_{k+1}1}\cd
\end{eqnarray*}
Consequently for any basic element $p \in KQ_+$ from $i$ to $j$\kd there exists $r \in e_iRe_j$ such that
\begin{eqnarray*}
\varphi(p) &=& \tanni{r}_{ij} \\
\varphi'(e(p)) &=& \tanni{r}_{in_i,j1}\cd
\end{eqnarray*}
Thus $\varphi(p)=0$ if and only if $r=0$ if and only if $\varphi'(e(p))$\cd

(3) This follows from the definition of $e$\cd 
\end{proof}

Now we complete the proof of Theorem \ref{main_thm2}\cd 
Put $I'=(e(\rho_1),\cdots,e(\rho_w))$\cd We shall show
$I'=\kernel{\varphi'}$\cd
We have $I' \subset \kernel{\varphi'}$ by Lemma \ref{lem_of_ext_map} (2)\cd

By Lemma \ref{lem_of_ext_map} (1)\kd we only have to show $\kernel{\varphi'} \cap \image{e} \subset I'$\cd
Assume that $p \in KQ_+$ satisfies $e(p) \in \kernel{\varphi'}$\cd
By Lemma \ref{lem_of_ext_map} (2)\kd we have $p \in \kernel{\varphi}$\cd We can write
\[
p=a_1\rho_1b_1+a_2\rho_2b_2+\cdots+a_w\rho_wb_w \quad (a_i,b_i \in KQ)\cd
\]
We have
\[
e(p)=e(a_1\rho_1b_1)+e(a_2\rho_2b_2)+\cdots+e(a_w\rho_wb_w)\cd
\]
By Lemma \ref{lem_of_ext_map} (3)\kd each $e(a_i \rho_i b_i)$ is contained in $I'$\cd 
Thus $p' \in I'$\cd
Consequently $I' = \kernel{\varphi'}$\cd  \hfill{$\square$}
\bigskip

\section{Application of main theorem to Harada algebras}
In this section\kd we describe quivers with relations of upper staircase factor algebras\cd
Then we can describe quivers with relations of Harada algebras completely\cd
We use the notations which were used in previous sections\cd

We assume that $R$ is a basic QF-algebra with Nakayama permutation $\sigma$\cd We put $P=R(n_1,\cdots,n_m)$ (see Definition \ref{block_ext})\cd
For $1 \leq i \leq m$ \kd we take a sequence
\[
1 \leq c_{i1} \leq c_{i2} \leq \cdots \leq c_{in_i} \leq n_i \kd
\]
Then we consider an upper staircase factor algebra $\overline{P}=P/X$\kd where $X$ is defined by above sequences (see  Definition \ref{upper_st_fac})\cd
For $1 \leq i \leq m$\kd we define $l_{i0},l_{i1},\cdots,l_{iu_i}$ to satisfy the following conditions\cd 
\begin{enumerate}
\def\labelenumi{(\theenumi)} 
\item $l_{i0}=0$ and $l_{iu_i}=n_i$\cd
\item $c_{i,l_{ij-1}+1}=\cdots=c_{il_{ij}}$ for $1 \leq j \leq u_i$\cd
\item $c_{il_{ij}} < c_{i,l_{ij}+1}$ for $1 \leq j \leq u_i$\cd
\end{enumerate}

Now let us calculate relations for $\overline{P}$\cd
We take a path $\theta_i$ in $Q$ from $i$ to $\sigma(i)$ such that $\varphi(\theta_i) \in \socle{e_iR}$\cd
We put
\begin{eqnarray*}
\theta'_i &=& \delta_i e(\theta_i) \delta_{\sigma(i)} 
= \delta_{i1} \cdots \delta_{i,n_i-1} e(\theta_i) \delta_{\sigma(i)1} \cdots \delta_{\sigma(i),n_{\sigma(i)}-1}
\end{eqnarray*}
for $1 \leq i \leq m$ (see Definition \ref{extension_map} for $\delta_i$ and $e$)\cd 
Then we have the following lemma\cd

\begin{lemma}
For $1 \leq i \leq m$\kd $\varphi'(\theta'_i) \in \socle{f_{i1}P}$\cd
\end{lemma}
\begin{proof}
By proof of Lemma \ref{lem_of_ext_map} (1) and Lemma \ref{lem_of_ext_map} (2)\kd we have $\varphi'(\theta'_i) \neq 0$\cd
Clearly $\varphi'(\theta'_i) \in f_{i1}P$\cd 
To show $\varphi'(\theta'_i) \in \socle{f_{i1}P}$\kd we only have to show that $\varphi'(\theta'_i) \jac{P}=0$\cd
For any $1 \leq t \leq d_{\sigma(i)s}$\kd we have by definition of the extension map
\[
\theta'_i \beta^t_{\delta(i)s}=\delta_i e(\theta) \delta_{\sigma(i)} \beta^t_{\delta(i)s}=\delta_i e(\theta_i\alpha^t_{\sigma(i)s})\cd
\]
Since $\varphi(\theta_i\alpha^t_{\sigma(i)s})=0$ and Lemma \ref{lem_of_ext_map} (2)\kd we have $\varphi'(e(\theta_i\alpha^t_{\sigma(i)s}))=0$\cd
Consequently $\varphi'(\theta'_i) \varphi'( \beta^t_{\delta(i)s})=\varphi'(\theta'_i \beta^t_{\delta(i)s})=0$\cd We have $\varphi'(\theta'_i) \jac{P}=0$\cd
\end{proof}

We put
\[
\theta'_i(u,v)=\begin{cases}
\delta_{iu} \cdots \delta_{in_i-1} e(\theta_i) \delta_{\sigma(i)1} \cdots \delta_{\sigma(i)v} 
& (1 \leq u \leq n_i-1 \kd 1 \leq v \leq n_{\sigma(i)}-1 ) \kd \\
e(\theta_i) \delta_{\sigma(i)1} \cdots \delta_{\sigma(i)v} & (u=n_i \kd 1 \leq v \leq n_{\sigma(i)}-1)\cd
\end{cases}
\]
Then the two-sided ideal $X$ of $P$ is generated by
\[
\varphi'(\theta'_{i}(l_{ij},c_{il_{ij}})) \quad \quad (1 \leq i \leq m \kd 1 \leq j \leq u_i) \kd
\]
since we can obtain any entry in $X(i,\sigma(i))$ by multiplying $\varphi'(\delta)$'s to one of these elements\cd
Consequently we have the following theorem\cd

\begin{thm}
Under the hypotheses as above\kd
we have a $K$-algebra isomorphism
\[
\overline{P} \simeq KQ'/I'
\]
for an ideal 
\[
I'=(e(\rho_1),\cdots,e(\rho_w))+(\theta'_{i}(l_{ij},c_{il_{ij}}) \ | \  1 \leq i \leq m \kd 1 \leq j \leq u_i)
\]
of $KQ'$\cd
\end{thm}

\begin{exam}
We consider $P=R(3,2)$ in Example \ref{exam_main_thm}\cd 
The Nakayama permutation of $R$ is the identity\cd 
We use the ideal $X$ in Example \ref{exam_of_str}\cd
\[
X=\left( \begin{array}{ccc|cc}
0 & (\alpha_{11})^3/(\alpha_{11})^4 & (\alpha_{11})^3/(\alpha_{11})^4 & 0 & 0 \\
0 & 0 & (\alpha_{11})^3/(\alpha_{11})^4 & 0 & 0 \\
0 & 0  & (\alpha_{11})^3/(\alpha_{11})^4 & 0 & 0 \\ \hline
0 & 0 & 0 & 0 & (\alpha_{21}\alpha_{12})/(\alpha_{21}\alpha_{12})^2 \\
0 & 0 & 0 & 0 & 0 
\end{array} \right) \cd
\]
In this case\kd $l_{11}=1 \kd l_{12}=3 \kd l_{21}=1 \kd l_{22}=2$\cd 
Then the quiver with relations of the upper staircase factor algebra $P/X$ is given by those in Example \ref{exam_main_thm} with additional relations
\[
\xymatrix{
& (1,2) \ar[rd]^{\delta_{12}} & \\
(1,1) \ar[ru]^{\delta_{11}} & & (1,3) \ar[d]^{\beta_{12}} \ar[ll]^{\beta_{11}}  \\
(2,2) \ar[u]^{\beta_{21}} & & (2,1) \ar[ll]_{\delta_{21}}
}
\quad \quad
\begin{cases}
\delta_{11} \delta_{12} e(\alpha_{11}^3) \delta_{11} = 0 \\
e(\alpha_{11}^3) \delta_{11} \delta_{12}=0 \\
\delta_{21} e(\alpha_{21} \alpha_{12}) \delta{21} = 0 
\end{cases}
\]
where $e(\alpha_{11}^3)=\beta_{11}\delta_{11}\delta_{12}\beta_{11}\delta_{11}\delta_{12}\beta_{11}$ and $e(\alpha_{21} \alpha_{12})=\beta_{12}\delta_{11}\delta_{12}\beta_{12}$\cd
\end{exam}


\bigskip


\begin{thebibliography}{99}
\bibitem{AF} F. W. Anderson, K. R. Fuller: Rings and Categories of Modules
(second edition), Graduate Texts in Math. 13, Springer-Verlag,
Heidelberg/New York/Berlin (1991)
\bibitem{ASS} I. Assem, D. Simson, A. Skowronski: Elements of the
Representation Theory of Associative Algebras, London Mathematical Society
Student Texts 65, Cambridge university press (2006)
\bibitem{AR}  M. Auslander, I. Reiten: $k$-Gorenstein algebras and syzygy modules, 
 J. Pure Appl. Algebra 92 (1994), 1-27.
\bibitem{ARS} M. Auslander, I. Reiten, S. Smalo: Representation Theory of
Artin Algebras, Cambridge Studies in Advanced Mathematics 36,
Cambridge university press  (1995)
\bibitem{quasi_H} Y. Baba and K. Iwase: On quasi-Harada rings, J. Algebra 155 (1996), 415-434
\bibitem{Lec_N} Y. Baba and K. Oshiro: Classical Artinian Rings and Related
Topics, Lecture note, preprint
\bibitem{FGR} R. Fossum, P. Griffith, I. Reiten: 
Trivial extensions of abelian categories, 
Lecture Notes in Mathematics Vol. 456. Springer-Verlag, Berlin-New York, (1975)
\bibitem{N_s_N_c} M. Harada: Non-small modules and non-cosmall modules, Ring
Theory. Proceedings of 1978 Antwerp Conference, New York (1979), 669-690
\bibitem{dua_H} K. Koike: Almost self-duality and Harada rings, J. Algebra 254 (2002), 336-361
\bibitem{LE_QF} K. Oshiro: Lifting modules, extending modules and their applications to QF-rings, Hokkaido Math. J. 13 
(1984), 310-338
\bibitem{LE_gur}K. Oshiro: Lifting modules, extending modules and their applications to generalized uniserial rings, 
Hokkaido Math. J. 13 (1984), 339-346
\bibitem{OHr1} K. Oshiro: On Harada rings I, Math. J. Okayama Univ. 31
(1989), 161-178
\bibitem{OHr2} K. Oshiro: On Harada rings II, Math. J. Okayama Univ. 31
(1989), 169-188
\bibitem{OHr3} K. Oshiro: On Harada rings III, Math. J. Okayama Univ. 32
(1990), 111-118
\bibitem{T} H. Tachikawa: Quasi-Frobenius rings and generalizations, 
Lecture Notes in Mathematics Vol. 351. Springer-Verlag, Berlin-New York, (1973)
\bibitem{Thr} Thrall, R. M.: Some generalization of quasi-Frobenius algebras,  Trans. Amer. Math. Soc. 64 
(1948), 173-183
\end{thebibliography}
\end{document}